\documentclass{article}
\usepackage{amsmath,amssymb,amsthm}
\usepackage{graphicx}

\newtheorem{thm}{Theorem}[section]
\newtheorem{cor}[thm]{Corollary}
\newtheorem{lem}[thm]{Lemma}

\newcommand{\F}{\mathcal{F}}

\newcommand{\la}{\left<}

\renewcommand{\phi}{\varphi}
\newcommand{\ra}{\right>}

\title{On 2-switches and isomorphism classes}
\author{Michael D. Barrus\\
School of Mathematics and Social Sciences, Black Hills State University\\
Spearfish, SD 57799, USA.\\
Email: michael.barrus@bhsu.edu}

\begin{document}
\maketitle
\begin{abstract}
A 2-switch is an edge addition/deletion operation that changes adjacencies in the graph while preserving the degree of each vertex. A well known result states that graphs with the same degree sequence may be changed into each other via sequences of 2-switches. We show that if a 2-switch changes the isomorphism class of a graph, then it must take place in one of four configurations. We also present a sufficient condition for a 2-switch to change the isomorphism class of a graph. As consequences, we give a new characterization of matrogenic graphs and determine the largest hereditary graph family whose members are all the unique realizations (up to isomorphism) of their respective degree sequences.\medskip

\noindent\emph{Keywords}: 2-switch, alternating 4-cycle, graph isomorphism, matrogenic graph, unigraph
\end{abstract}

\section{Introduction}\label{sec: intro}
The degree sequence is one of the simplest parameters associated with a graph, and the efficiency of several graph algorithms depends upon this simplicity. However, the utility of a degree sequence in a graph problem is limited by the fact that most degree sequences belong to several pairwise nonisomorphic graphs (the \emph{realizations} of their respective degree sequences). It is desirable, then, to understand the relationships that exist among graphs with the same degree sequence.

A well known result of Fulkerson, Hoffman, and McAndrew~\cite{FulkersonEtAl65} links graphs having the same degree sequence. An \emph{alternating 4-cycle} is an instance of four vertices $a,b,c,d$ in a graph $G$ such that $ab$ and $cd$ are edges of $G$ and $bc$ and $ad$ are not; we denote this alternating 4-cycle by $\la a,b:c,d\ra$. (See Figure~\ref{fig: a4}.) Suppose $G$ has such an alternating 4-cycle, let $H$ be the graph obtained by deleting $ab$ and $cd$ from $G$ and adding edges $bc$ and $ad$. We refer to these operations as a \emph{2-switch (on $\la a,b:c,d\ra$)}. Note that a 2-switch leaves the degree of every vertex unchanged, so $G$ and $H$ have the same degree sequence.
\begin{figure}
\centering
%\begin{pspicture}(0,0)(1,1)
%\cnode*(0,1){3pt}{a} \cnode*(1,1){3pt}{b} \cnode*(1,0){3pt}{c} \cnode*(0,0){3pt}{d}
%\ncline{-}{d}{a} \ncline{-}{a}{e} \ncline{-}{e}{b} \ncline{-}{b}{c}
%\ncline[linestyle=dotted]{-}{a}{b} \ncline[linestyle=dotted]{-}{c}{d}
%\end{pspicture}
\includegraphics[width=1.4cm]{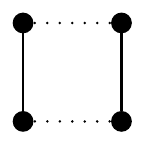}
\caption{An alternating 4-cycle.}
\label{fig: a4}
\end{figure}

\begin{thm}[\cite{FulkersonEtAl65}] \label{thm: FHM}
Two unlabeled graphs $G$ and $H$ have the same degree sequence if and only if there is a sequence of 2-switches that transforms $G$ into $H$.
\end{thm}

Since a 2-switch changes the adjacencies in a graph, it seems natural that it should often change the isomorphism class of the graph. However, this is not always the case; for example, the unique 2-switch possible on a four-vertex path creates another four-vertex path. In Section 2, we study when a 2-switch changes the isomorphism class of a graph. We present a necessary condition, showing that a 2-switch cannot change the isomorphism class of a graph unless it occurs within one of four vertex-edge configurations (we give a precise definition of a configuration in the following section). We further show that this necessary condition is also sufficient unless the 2-switch happens in one of two more configurations, which are produced by overlaying two of the original four configurations in a ``symbiotic'' way.

We apply these results in Sections 3 and 4 to special classes of graphs. In Section 3 we discuss the matrogenic graphs. We give a new proof that matrogenic graphs are \emph{unigraphs}, that is, unique realizations of their respective degree sequences up to isomorphism. We also provide a new characterization of matrogenic graphs in terms of their alternating 4-cycles and automorphisms.

In Section 4 we study hereditary unigraphs. Though the class of unigraphs is not hereditary (closed under taking induced subgraphs), a number of interesting hereditary graph classes, among them the threshold graphs and matrogenic graphs, contain only unigraphs. Hereditary classes of unigraphs also appear to have an important role in the study of hereditary graph classes having degree sequence characterizations (see \cite{MinTriples,BarrusEtAl08,NonMinimalTriples}). We characterize all hereditary classes of unigraphs, including the maximum such class, in terms of their minimal forbidden induced subgraphs.

\section{2-switches and configurations} \label{sec: 2-switches}
A \emph{configuration} is a triple $(V,E,F)$ such that $V$ and $E$ are the vertex and edge sets, respectively, of a graph $G$, and $F$ is a subset of the edge set of the complement of $G$. We call the elements of $F$ ``non-edges.'' The alternating 4-cycle is an example of a configuration. As in Figure~\ref{fig: a4}, we represent the configuration visually by drawing the graph $G$ with solid edges and joining the vertex pairs in $F$ with dotted segments.

A graph $H$ \emph{contains} the configuration $(V,E,F)$ on vertex set $W$ if there exists a bijection $f:V \to W$ such that vertex pairs in $V$ belonging to $E$ are mapped to adjacent vertices in $H$, and vertex pairs in $V$ belonging to $F$ are mapped to nonadjacent vertices in $H$.

Let $V(G)$ and $E(G)$ denote the vertex set and edge set, respectively, of $G$. A \emph{module} $M$ in $G$ is a subset of $V(G)$ with the property that every vertex in $V(G) \setminus M$ is adjacent to either all or none of the vertices in $M$.

\begin{lem}\label{lem: module implies iso}
Let $\la a,b:c,d\ra$ be an alternating 4-cycle in a graph $G$, and let $H$ be the graph obtained from $G$ by performing the 2-switch on $\la a,b:c,d \ra$. If $\{a,c\}$ is a module in $G-\{b,d\}$, or if $\{b,d\}$ is a module in $G-\{a,c\}$, then $G \cong H$.
\end{lem}
\begin{proof}
For $u,v \in V(G)$, let $\phi_{uv}$ denote the function on $V(G)$ that maps $u$ and $v$ to each other and fixes every other element of $V(G)$. We claim that if $\{a,c\}$ is a module in $G-\{b,d\}$, then $\phi_{ac}: V(G) \to V(H)$ is an isomorphism. Indeed, both the 2-switch and $\phi_{ac}$ leave unchanged the relationships (adjacency or nonadjacency) between pairs of vertices that contains neither $a$ nor $c$. Let $u$ be any vertex of $G$ not in $\{a,c\}$. By the definition of the 2-switch on $\la a,b:c,d \ra$, and since $\{a,c\}$ is a module in $G-\{b,d\}$, we know that $ua \in E(G)$ if and only if $uc \in E(H)$, if and only if $\phi_{ac}(u)\phi_{ac}(a) \in E(H)$. Similarly, $uc \in E(G)$ if and only if $\phi_{ac}(u)\phi_{ac}(c) \in E(H)$. Finally, the pair $\{a,c\}$ is mapped to itself by $\phi_{ac}$ and unchanged by the 2-switch.

Similarly, if $\{b,d\}$ is a module in $G - \{a,c\}$, then $\phi_{bd}:V(G) \to V(H)$ is an isomorphism.
\end{proof}

\begin{thm}\label{thm: If 2-switch then config}
If a graph $G$ admits a 2-switch that changes the isomorphism class of the graph, then $G$ contains one of the configurations in Figure~\ref{fig: configs}, with vertices $p,q,r,s$ as marked, such that the 2-switch is on $\la p,q:r,s\ra$.
\end{thm}
\begin{figure}
\centering
\includegraphics[height=2.2cm]{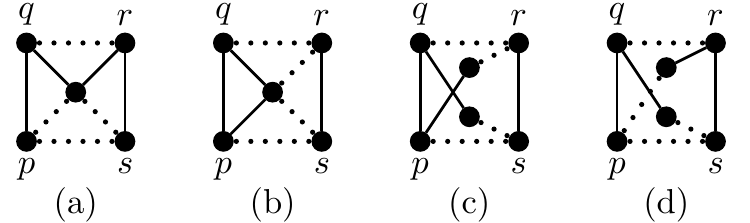}
\caption{Configurations for isomorphism class-changing 2-switches.}
\label{fig: configs}
\end{figure}
\begin{proof}
Suppose that the isomorphism class-changing 2-switch is on $\la a,b:c,d \ra$. By Lemma~\ref{lem: module implies iso}, there must be vertices $u,v$ (with possibly $u=v$) outside of $\{a,b,c,d\}$ such that $u$ is adjacent to exactly one of $a$ and $c$, and $v$ is adjacent to exactly one of $b$ or $d$. Depending on whether $u=v$ and to which of $\{a,c\}$ and $\{b,d\}$, respectively, $u$ and $v$ are adjacent to, $G$ must contain one of the four configurations on the vertex set $\{a,b,c,d,u,v\}$ (with $\{p,q,r,s\}=\{a,b,c,d\}$).
\end{proof}

We observe that in performing the 2-switch on $\la p,q:r,s \ra$, the configuration in (a) yields a configuration identical (save for the vertex labels) to that in (b), and the 2-switch on $\la p,q:r,s \ra$ in the configuration in (b) produces a configuration matching that in (a). The configurations in (c) and (d) also yield unlabeled copies of each other after a 2-switch. Furthermore, replacing edges with non-edges and vice versa in each configuration above, we see that $G$ contains the configuration in (a) or (c) if and only if $\overline{G}$ contains the configuration in (b) or (d), respectively.

An example shows that the converse of Theorem~\ref{thm: If 2-switch then config} is not true. Let $U$ be the graph shown in Figure~\ref{fig: U}; up to isomorphism, this is the unique graph having degree sequence $(4,2,2,2,2,2)$.
\begin{figure}
\centering
%\begin{pspicture}(0,-0.1)(1.5,1.1)
%\cnode*(0,0){3pt}{u} \cnode*(0,1){3pt}{v} \cnode*(0.5,0.5){3pt}{w} \cnode*(1,0){3pt}{x} \cnode*(1,1){3pt}{y} \cnode*(1.5,0.5){3pt}{z}
%\uput[d](0,0){$u$}\uput[u](0,1){$v$}\uput[d](0.5,0.5){$w$}\uput[d](1,0){$x$} \uput[u](1,1){$y$} \uput[r](1.5,0.5){$z$}
%\ncline{-}{u}{v} \ncline{-}{u}{w} \ncline{-}{v}{w} \ncline{-}{w}{x} \ncline{-}{w}{y} \ncline{-}{x}{z} \ncline{-}{y}{z}
%\end{pspicture}
\includegraphics[width=1.9cm]{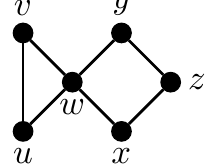}
\caption{The unigraph $U$.}
\label{fig: U}
\end{figure}
The graph contains the configuration in Figure~\ref{fig: configs}(a) on vertex set $\{v,w,x,y,z\}$, but since it is a unigraph, the 2-switch on $\la v,w:z,x \ra$ does not change the isomorphism class. From the paragraph above, we conclude that $U$ must also contain the configuration in Figure~\ref{fig: configs}(b), which it does, on vertex set $\{u,v,w,x,z\}$.

As mentioned above, the 2-switch on $\la p,q,r,s \ra$ changes a copy of the configuration in Figure~\ref{fig: configs}(a) into a copy of the configuration in (b). (Observe that no 5-vertex subgraph can simultaneously contain both configurations.) If the 2-switch does not change the isomorphism class of the graph $G$ containing the configuration, then the number of copies in $G$ of each configuration must remain the same after the 2-switch; hence the 2-switch must also destroy some copy in $G$ of the configuration in (b) and create a copy of the configuration in (a). Similar remarks hold for the configurations in Figure~\ref{fig: configs}(c) and (d).

Intuitively, it seems that the simplest way to preserve the number of copies in $G$ of the configurations in (a) and (b) during a 2-switch is to ``overlay'' the configurations, identifying the copies of $p$, $q$, $r$, and $s$ while keeping the central vertices distinct. In this way, the 2-switch will simultaneously create a copy of one of these configurations on one set of five vertices at the same time it eliminates a copy of the same configuration on another set of five vertices; the two overlaid configurations have a ``symbiotic'' relationship, working together to keep their numbers equal before and after the 2-switch. We show that in fact this phenomenon, or an analogous version involving the configurations in (c) and (d), \emph{must} happen when a 2-switch fails to change an isomorphism class.

\begin{thm} \label{thm: nonchange config}
If a graph $G$ contains one of the configurations shown in Figure~\ref{fig: configs}, and the 2-switch on $\la p,q:r,s \ra$ produces a graph isomorphic to $G$, then the vertices of the configuration lie in one of the configurations in Figure~\ref{fig: configs II}, where the alternating cycle $\la u,v:w,x \ra$ coincides with $\la p,q:r,s \ra$.
\end{thm}
\begin{figure}
\centering
%\begin{pspicture}(0,-0.675)(4.5,1.5)
%\psset{unit=1.5cm}
%\cnode*(0,1){3pt}{v} \cnode*(1,1){3pt}{w} \cnode*(1,0){3pt}{x} \cnode*(0,0){3pt}{u}
%\cnode*(0.5,0.3){3pt}{y} \cnode*(0.5,0.7){3pt}{z}
%\uput[d](0,0){$u$}\uput[u](0,1){$v$}\uput[u](1,1){$w$}\uput[d](1,0){$x$}
%\ncline{-}{u}{v} \ncline{-}{u}{y} \ncline{-}{y}{v} \ncline{-}{v}{z} \ncline{-}{z}{w} \ncline{-}{w}{x}
%\ncline[linestyle=dotted,linewidth=1.5pt]{-}{y}{x} \ncline[linestyle=dotted,linewidth=1.5pt]{-}{u}{z}
%\ncline[linestyle=dotted,linewidth=1.5pt]{-}{y}{w} \ncline[linestyle=dotted,linewidth=1.5pt]{-}{x}{z}
%\ncline[linestyle=dotted,linewidth=1.5pt]{-}{u}{x}\ncline[linestyle=dotted,linewidth=1.5pt]{-}{v}{w}
%\uput[d](0.5,-0.3){(a)}
%%
%\cnode*(2,0){3pt}{u2} \cnode*(2,1){3pt}{v2} \cnode*(3,1){3pt}{w2} \cnode*(3,0){3pt}{x2}
%\cnode*(2.5,0.15){3pt}{y2} \cnode*(2.5,0.85){3pt}{z2} \cnode*(2.5,0.5){3pt}{t2}
%\uput[d](2,0){$u$}\uput[u](2,1){$v$}\uput[u](3,1){$w$}\uput[d](3,0){$x$}
%\ncline{-}{u2}{v2} \ncline{-}{w2}{x2} \ncline{-}{u2}{y2} \ncline{-}{v2}{t2} \ncline{-}{w2}{z2}
%\ncline[linestyle=dotted,linewidth=1.5pt]{-}{v2}{w2} \ncline[linestyle=dotted,linewidth=1.5pt]{-}{t2}{x2} \ncline[linestyle=dotted,linewidth=1.5pt]{-}{u2}{z2} \ncline[linestyle=dotted,linewidth=1.5pt]{-}{y2}{w2} \ncline[linestyle=dotted,linewidth=1.5pt]{-}{x2}{u2}
%\uput[d](2.5,-0.3){(b)}
%\end{pspicture}
\includegraphics[height=2.8cm]{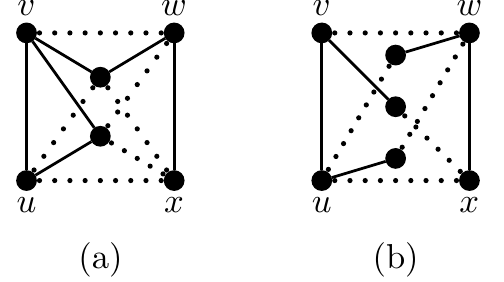}
\caption{The configurations from Theorem~\ref{thm: nonchange config}.}
\label{fig: configs II}
\end{figure}
\begin{proof}
Suppose first that $G$ contains a configuration from Figure~\ref{fig: configs}(a) or (b), with vertices $p,q,r,s$ as labeled there, and the 2-switch on $\la p,q:r,s \ra$ produces a graph isomorphic to $G$. Let $A$ denote the set of all vertices in $G$ whose neighborhoods have intersection with $\{p,q,r,s\}$ equal to $\{p,q\}$ or $\{r,s\}$; let $B$ denote the set of all vertices whose neighborhoods have intersection either $\{q,r\}$ or $\{p,s\}$. Let $C$ and $D$ denote the sets of all vertices whose neighborhoods intersect $\{p,q,r,s\}$ in exactly three and four vertices, respectively. Since the graph resulting from the 2-switch is isomorphic to $G$, the total number of triangles destroyed by the edge deletions in the 2-switch must equal the number of triangles created by the edge additions in the 2-switch. The former quantity equals $|A|+|C|+2|D|$, while the latter equals $|B|+|C|+2|D|$, so we conclude that $|A|=|B|$. It is clear from the configuration that $A$ or $B$ is nonempty, so in fact both are nonempty, and the five vertices of the configuration belong to the configuration shown in Figure~\ref{fig: configs II}(a), with $\la p,q:r,s \ra$ equaling $\la u,v:w,x \ra$ (we avoid equating vertices between the two alternating 4-cycles to allow for symmetries).

Suppose instead that $G$ contains the configuration from Figure~\ref{fig: configs}(c), with the vertices as labeled in the figure, and vertices $y$ and $z$ the vertices shown to be adjacent to $q$ and $p$, respectively. If $G$ contains a seventh vertex $t$ that is adjacent to $r$ but not to $p$, or to $s$ but not to $q$, then these vertices lie in the configuration shown in Figure~\ref{fig: configs II}(b) with $\la p,q:r,s \ra$ equaling $\la u,v:w,x \ra$, so let us suppose henceforth that any vertex outside $\{p,q,r,s,y,z\}$ adjacent to $r$ (respectively, $s$) must be adjacent to $p$ (to $q$).

If $r$ and $p$ have the same degree in $G$, then $y$ is adjacent to $r$ but not $p$, and any vertex outside $\{p,q,r,s,y,z\}$ adjacent to $p$ is also adjacent to $r$. The induced subgraph with vertex set $\{p,q,r,s,y\}$ contains the configuration from Figure~\ref{fig: configs}(a), so by the argument above, $G$ contains the configuration in Figure~\ref{fig: configs II}(a) with vertex set $\{p,q,r,s,y,t\}$ for some $t$. Note that $t$ must be adjacent to exactly one of $r$ and $p$; we conclude that $t=z$. Hence if the degrees of $p$ and $r$ are equal, then the theorem holds. A similar argument with $q,s,z$ replacing $p,r,s$ shows that the theorem also holds if the degrees of $q$ and $s$ in $G$ are equal.

Suppose instead that $d_G(p)>d_G(r)$ and $d_G(q)>d_G(s)$. We define the \emph{neighborhood list} of a vertex $a$ in $G$ to be the multiset $\{d_G(b): b\text{ is adjacent to }a\}$ for $a \in V(G)$, and we denote this by $L_G(a)$. The \emph{neighborhood list of $G$} will be the multiset of neighborhood lists of all its vertices. It is clear that the neighborhood list is a graph invariant. Furthermore, the only vertices whose neighborhood lists can possibly change during a 2-switch are the four vertices of the alternating 4-cycle involved. Now let $G'$ denote the graph obtained from $G$ via the 2-switch on $\la p,q:r,s \ra$. Since $G'$ is isomorphic to $G$, the two graphs have the same neighborhood list, and since $L_{G}(a) = L_{G'}(a)$ for all $a \in V(G)\setminus \{p,q,r,s\}$, we have $\{L_G(p), L_G(q), L_G(r),L_G(s)\} = \{L_{G'}(p), L_{G'}(q), L_{G'}(r), L_{G'}(s)\}$. The multisets $L_G(p)$ and $L_{G'}(p)$ have the same size (the degree of $p$) and differ in exactly one element; in $L_{G'}(p)$ the term $d_G(q)$ in $L_G(p)$ is replaced by $d_G(s)$, which is smaller. Similar arguments show that neighborhood lists of $q$, $r$, and $s$ are also changed by the 2-switch. Now let $a$ be a vertex in $\{p,q,r,s\}$ having a largest neighborhood list in $G$; we have $a \in \{p,q\}$. Since $a$ is also a largest neighborhood list in $G'$, we conclude that $d_G(p)=d_G(q)$. Comparing lengths of neighborhood lists, we are forced to conclude that $L_{G'}(p) = L_G(q)$ and $L_{G'}(q) = L_G(p)$. As noted above, $L_{G'}(p)$ can be obtained from $L_{G}(p)$, and hence from $L_{G'}(q)$, by replacing the term $d_G(q)$ by the smaller number $d_G(s)$. However, a similar argument shows that $L_{G'}(q)$ can be obtained by replacing a term of $L_{G'}(p)$ by a smaller number; together these statements form a contradiction. Thus the theorem holds if $G$ contains the configuration from Figure~\ref{fig: configs}(c).

Finally, if $G$ contains the configuration from Figure~\ref{fig: configs}(d), with the vertices as labeled in the figure, then the complement $\overline{G}$ contains the configuration in Figure~\ref{fig: configs}(c), albeit with the vertex labels in different places. If $G'$ denotes the graph resulting from the 2-switch on $\la p,q:r,s\ra$ in $G$, then $\overline{G'}$ is precisely the graph obtained from $\overline{G}$ via the 2-switch on $\la q,r:s,p\ra$. By the arguments above, the six vertices of the configuration lie, in $\overline{G}$, in one of the configurations shown in Figure~\ref{fig: configs II}, with $\la q,r:s,p \ra$ corresponding to $\la u,v:w,x \ra$. The complement of this latter configuration yields the same configuration with the same vertex set in $G$, with the alternating 4-cycle $\la u,v:w,x\ra$ in $G'$ corresponding to an alternating 4-cycle on the same vertices in $G$. Thus, the theorem holds in this last case as well.
\end{proof}

\begin{cor}
The converse of Theorem~\ref{thm: If 2-switch then config} is true for graphs that contain neither of the configurations shown in Figure~\ref{fig: configs II}.
\end{cor}

\section{Matrogenic graphs}
Matrogenic graphs are those graphs $G$ where the vertex sets of the alternating 4-cycles of $G$ form the circuits of a matroid on $V(G)$. These graphs were introduced in \cite{FoldesHammer78} and shown to be the graphs that forbid the configuration in Figure~\ref{fig: matr forb config}. %
\begin{figure}
\centering
%\begin{pspicture}(0,0)(1.5,1.5)
%\psset{unit=0.75cm}
%\cnode*(1,0){3pt}{a} \cnode*(0,1){3pt}{b} \cnode*(1,1){3pt}{c} \cnode*(2,1){3pt}{d} \cnode*(1,2){3pt}{e}
%\ncline[linestyle=dotted,linewidth=1.5pt]{-}{a}{b} \ncline[linestyle=dotted,linewidth=1.5pt]{-}{a}{d} \ncline[linestyle=dotted,linewidth=1.5pt]{-}{c}{e}
%\ncline{-}{a}{c} \ncline{-}{b}{e} \ncline{-}{d}{e}
%\end{pspicture}
\includegraphics[width=1.9cm]{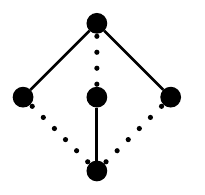}
\caption{The forbidden configuration for matrogenic graphs.}
\label{fig: matr forb config}
\end{figure}
Since each of the configurations in Figure~\ref{fig: configs} contains this configuration (on the four outer vertices and one interior vertex), Theorem 2.2 yields another proof of the following:

\begin{cor}[\cite{MarchioroEtAl84,Tyshkevich84}]
Matrogenic graphs are unigraphs.
\end{cor}

Furthermore, the results in Section 2 allow us to give new characterizations of matrogenic graphs. As in the proof of Lemma~\ref{lem: module implies iso}, given a set $T$ and elements $u$ and $v$ of the set, let $\phi_{uv}$ denote the bijection from $T$ to itself that maps $u$ and $v$ to each other and fixes all other elements of $T$.
\begin{lem} \label{lem: matr char}
A graph $G$ is matrogenic if and only if the following holds: If $\la a,b:c,d \ra$ is any alternating 4-cycle in $G$, and $H$ is the graph obtained from $G$ by performing the 2-switch on $\la a,b:c,d\ra$, then $\phi_{ac}$ and $\phi_{bd}$ are both isomorphisms from $V(G)$ to $V(H)$.
\end{lem}
\begin{proof}
Suppose $G$ is matrogenic. Since $G$ forbids the configuration in Figure~\ref{fig: matr forb config}, the sets $\{a,c\}$ and $\{b,d\}$ are modules in $G - \{b,d\}$ and $G - \{a,c\}$, respectively, so as in the proof of Lemma~\ref{lem: module implies iso} both $\phi_{ac}$ and $\phi_{bd}$ are isomorphisms.

Conversely, if $G$ is not matrogenic, then it contains the configuration in Figure~\ref{fig: matr forb config} and hence an alternating 4-cycle $\la a,b:c,d \ra$ and a vertex $u$ other than $b$ that is adjacent to $a$ but not to $c$. Let $H$ be the graph resulting from the 2-switch on $\la a,b:c,d \ra$. The map $\phi_{ac}:V(G) \to V(H)$ is not an isomorphism, since $ua \in E(G)$, but $\phi_{ac}(u)\phi_{ac}(a) = uc \notin E(H)$.
\end{proof}

\begin{thm}
A graph is matrogenic if and only if for each alternating 4-cycle $\la a,b:c,d \ra$ in $G$, the map from $V(G)$ to itself that transposes $a$ and $c$, transposes $b$ and $d$, and fixes every other vertex is a graph automorphism.
\end{thm}
\begin{proof}
Let $\rho:V(G) \to V(G)$ denote the map described. Note that $\rho = \phi_{bd}^{-1}\phi_{ac}$. If $G$ is matrogenic and $\la a,b:c,d \ra$ and $H$ are as in Lemma~\ref{lem: matr char}, then both $\phi_{ac}$ and $\phi_{bd}$ are isomorphisms, so $\rho$ is an automorphism. If $G$ is not matrogenic, then as in the proof of Lemma~\ref{lem: matr char}, there exists an alternating 4-cycle $\la a,b:c,d\ra$ and a vertex $u$ not in $\{b,d\}$ such that $ua$ is an edge in $G$ and $uc$ is not. Since $\rho(u)\rho(a) = uc$, the map $\rho$ is not an automorphism.
\end{proof}

\section{Hereditary classes of unigraphs} \label{sec: unigraphs}

As mentioned in the introduction, a graph class is hereditary if it is closed under taking induced subgraphs. Naively, one might assume that the class of unigraphs is hereditary, since if some induced subgraph $S$ of a graph $G$ were not a unigraph, then performing a 2-switch on $S$ to change the isomorphism class of that subgraph would seem likely to change the isomorphism class of $G$ as well. However, the graph $U$ in Figure~\ref{fig: U} is an example of a unigraph containing non-unigraphs as induced subgraphs; $U$ induces two nonisomorphic subgraphs having degree sequence $(3,2,2,2,1)$.

Still, there are several familiar hereditary classes that contain only unigraphs. The matrogenic graphs discussed in Section 3 are an example. The class of threshold graphs is another. Though they have several equivalent definitions, threshold graphs were shown in \cite{ChvatalHammer73} to be precisely the class of graphs containing no alternating 4-cycle. Because every 2-switch is performed on an alternating 4-cycle, it follows from Theorem~\ref{thm: FHM} that threshold graphs are unigraphs. (For a survey of results on both threshold and matrogenic graphs, see~\cite{MahadevPeled95}.)

Given a set $\F$ of graphs, we say a graph $G$ is \emph{$\F$-free} if no induced subgraph of $G$ is isomorphic to an element of $\F$. For every hereditary graph class $\mathcal{H}$ there is a set $\F$ of graphs such that $\mathcal{H}$ is precisely the class of $\F$-free graphs. We call elements of $\F$ \emph{forbidden subgraphs} for $\mathcal{H}$. 

Forbidding a configuration such as the one in Figure~\ref{fig: a4} or Figure~\ref{fig: matr forb config} from a graph is equivalent to forbidding a set of induced subgraphs. If a configuration $\mathcal{C}$ has $k$ vertices, let $\mathcal{R}(\mathcal{C})$ denote all $k$-vertex graphs containing the configuration $\mathcal{C}$. We call the elements of $\mathcal{R}(\mathcal{C})$ the \emph{realizations} of $\mathcal{C}$. Clearly a graph is $\mathcal{R}(\mathcal{C})$-free if and only if it does not contain the configuration $\mathcal{C}$. For example, threshold graphs are the $\{2K_2,C_4,P_4\}$-free graphs~\cite{ChvatalHammer73}; these forbidden subgraphs are the realizations of the alternating 4-cycle. The forbidden configuration for matrogenic graphs has ten corresponding forbidden induced subgraphs.

In a recent paper~\cite{BarrusEtAl08}, the author and others defined the set $\F$ of forbidden induced subgraphs to be \emph{degree-sequence-forcing (DSF)} if the membership of a graph $G$ in the class $\mathcal{H}$ can be determined knowing only the degree sequence $G$. Thus $\{2K_2,C_4,P_4\}$ and the set of forbidden subgraphs for the matrogenic graphs are both DSF sets. Among the few sufficient conditions known for a set to be DSF is the following:
\begin{thm}[\cite{BarrusEtAl08}] \label{thm: DSF unigraph lemma}
If the $\F$-free graphs are all unigraphs, then $\F$ is a DSF set.
\end{thm}
Such ``unigraph-producing'' sets $\F$ are not the only DSF sets, but empirical evidence suggests~\cite{MinTriples,BarrusEtAl08,NonMinimalTriples} that they account for a large portion of the DSF sets containing small graphs.

We call a graph a \emph{hereditary unigraph} if every one of its induced subgraph is a unigraph. Since the union of a collection of hereditary graph classes is again a hereditary class of graphs, there is a maximum hereditary class of unigraphs, the set of all hereditary unigraphs. Using the results of Section~\ref{sec: 2-switches}, we now characterize this class, which will properly contain the matrogenic graphs and threshold graphs. As a corollary, we derive a characterization of all sets satisfying the hypothesis of Theorem~\ref{thm: DSF unigraph lemma}.

The \emph{house} graph is the complement of $P_5$. The \emph{4-pan} is the graph obtained by attaching a pendant vertex to one vertex of a 4-cycle; the complement of the 4-pan is called the co-4-pan. Let $R$ and $S$ be as shown in Figure~\ref{fig: R and S}. %
\begin{figure}
\centering
%\includegraphics[width=5cm]{ck_fig2.tif}
%\begin{pspicture}(9.25,-0.6)(19.5,1)
%% R
%\cnode*(9.25,1){3pt}{W} \cnode*(10.25,1){3pt}{Z} \cnode*(11.25,1){3pt}{AA} \cnode*(9.25,0){3pt}{BB} \cnode*(10.25,0){3pt}{CC} \cnode*(11.25,0){3pt}{DD}
%\ncline{-}{W}{Z} \ncline{-}{Z}{AA} \ncline{-}{Z}{CC} \ncline{-}{Z}{DD} \ncline{-}{AA}{DD} \ncline{-}{BB}{CC} \ncline{-}{CC}{DD}
%\uput[d](10.25,-0.125){\fontsize{7pt}{7pt}$R$}
%% R-bar
%\cnode*(12,1){3pt}{EE} \cnode*(13,1){3pt}{FF} \cnode*(14,1){3pt}{GG} \cnode*(12,0){3pt}{HH} \cnode*(13,0){3pt}{II} \cnode*(14,0){3pt}{JJ}
%\ncline{-}{EE}{FF} \ncline{-}{EE}{II} \ncline{-}{FF}{GG} \ncline{-}{FF}{II} \ncline{-}{FF}{JJ} \ncline{-}{GG}{JJ} \ncline{-}{HH}{II} \ncline{-}{HH}{JJ}
%\uput[d](13,-0.05){\fontsize{7pt}{7pt}$\overline{R}$}
%% S
%\cnode*(14.75,0){3pt}{AA} \cnode*(14.75,1){3pt}{AA2} \cnode*(15.42,0.5){3pt}{BB} \cnode*(16.08,1){3pt}{CC} \cnode*(16.08,0){3pt}{DD} \cnode*(16.75,0.5){3pt}{ZZ}
%\ncline{-}{AA}{BB} \ncline{-}{AA2}{BB} \ncline{-}{BB}{CC} \ncline{-}{BB}{DD} \ncline{-}{CC}{DD} \ncline{-}{CC}{ZZ} \ncline{-}{DD}{ZZ}
%\uput[d](15.75,-0.05){\fontsize{7pt}{7pt}$S$}
%% S-bar
%\cnode*(17.5,0.5){3pt}{EE} \cnode*(18.17,0.5){3pt}{FF} \cnode*(18.83,1){3pt}{GG} \cnode*(18.83,0){3pt}{HH} \cnode*(19.5,0){3pt}{II} \cnode*(19.5,1){3pt}{II2}
%\ncline{-}{EE}{FF} \ncline{-}{FF}{GG} \ncline{-}{FF}{HH} \ncline{-}{GG}{HH} \ncline{-}{GG}{II} \ncline{-}{HH}{II} \ncline{-}{GG}{II2} \ncline{-}{HH}{II2}
%\uput[d](18.5,-0.05){\fontsize{7pt}{7pt}$\overline{S}$}
%\end{pspicture}
\includegraphics[width=10.2cm]{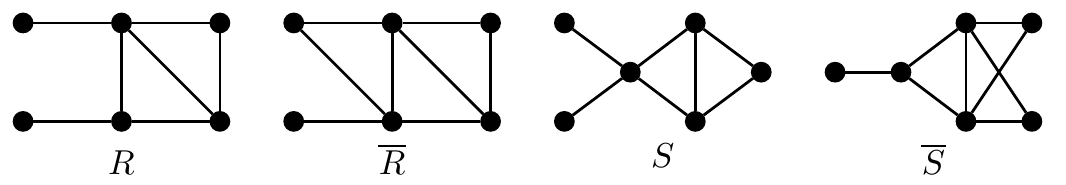}
\caption{The graphs $R$, $\overline{R}$, $S$, and $\overline{S}$.}
\label{fig: R and S}
\end{figure}
Henceforth let \begin{multline}\nonumber\F = \{P_5,\text{house},K_2+K_3,K_{2,3},\text{4-pan},\text{co-4-pan},2P_3,(K_2+K_1)\vee(K_2+K_1),\\K_2+P_4,2K_1 \vee P_4,K_2+C_4,2K_1 \vee 2K_2, R, \overline{R}, S, \overline{S}\},\end{multline} where $+$ and $\vee$ indicate disjoint union and join, respectively.

\begin{thm}\label{thm: forb subgr for hered unigraphs}
The following are equivalent for a graph $G$:
\begin{enumerate}
\item[\textup{(i)}] $G$ is a hereditary unigraph;
\item[\textup{(ii)}] $G$ contains none of the configurations in Figure~\ref{fig: configs};
\item[\textup{(iii)}] $G$ is $\F$-free.
\end{enumerate}
\end{thm}
\begin{proof}
(i) $\iff$ (ii): By inspection, we see that none of the realizations of any of the configurations in Figure~\ref{fig: configs} is a unigraph; hence hereditary unigraphs cannot contain any of these configurations. If $G$ is not a hereditary unigraph, then for some induced subgraph $J$ of $G$ there is a 2-switch possible that changes the isomorphism class of the subgraph induced on $V(J)$. By Theorem~\ref{thm: If 2-switch then config}, $J$ (and hence $G$) must contain a configuration from Figure~\ref{fig: configs}.

(ii) $\implies$ (iii): The graphs $P_5$, $K_{2,3}$, and the 4-pan are all realizations of the configuration in Figure~\ref{fig: configs}(a), and their complements are all realizations of Figure~\ref{fig: configs}(b). Each of $K_2+P_4$, $K_2+C_4$, $(K_2+K_1)\vee(K_2+K_1)$, $R$, and $\overline{S}$ is a realization of the configuration in Figure~\ref{fig: configs}(c), and their complements are realizations of the configuration in Figure~\ref{fig: configs}(d). Thus if $G$ contains none of the configurations in Figure~\ref{fig: configs}, then $G$ is $\F$-free.

(iii) $\implies$ (ii): Let $G$ be an arbitrary $\F$-free graph. Since $\{P_5,K_{2,3},\text{4-pan}\}$ and $\{\text{house},K_2+K_3,\text{co-4-pan}\}$ are the sets of realizations of the configurations in Figures~\ref{fig: configs}(a) and (b), respectively, $G$ contains neither of these configurations. Suppose that $G$ contains the configuration shown in (c), and let $p,q,r,s$ be vertices of $G$ as labeled in the figure, with $t$ and $u$ denoting the other configuration vertices adjacent to $p$ and to $q$, respectively. Since $G$ is $\{K_2+K_3,K_2+P_4,K_2+C_4\}$-free, either $r$ or $s$ must have a neighbor in $\{p,q,t,u\}$; by symmetry, we assume that $r$ is adjacent to either $p$ or $u$.

If $r$ is adjacent to $u$, then to avoid a copy of $P_5$, the 4-pan, or $K_{2,3}$ with vertex set $\{p,q,r,s,u\}$, we must have $p$ adjacent to $u$. To avoid a copy of the house or co-4-pan with vertex set $\{p,q,r,s,u\}$, we must have $r$ adjacent to $p$ as well. Since $p$ and $u$ have exactly the same neighbors and and nonneighbors among $\{q,s\}$ in the original configuration, the same argument, with $p$ and $u$ exchanging places, shows that if $r$ is adjacent to $p$ then it is adjacent to $u$. We therefore conclude that $\{p,r,u\}$ is a triangle in $G$.

Consider the vertex set $\{p,q,r,s,t\}$. If $qt$ is not an edge, then in order to avoid a copy of the 4-pan or $K_{2,3}$ on this vertex set, neither $st$ nor $qs$ is an edge. However, then the subgraph of $G$ induced on $\{p,q,r,s,t,u\}$ is isomorphic to either $R$ or $\overline{S}$, depending on whether $tu$ is an edge, a contradiction. Thus $qt$ is an edge. To avoid a copy of the co-4-pan or the house graph with vertex set $\{p,q,r,s,t\}$, vertex $s$ must be adjacent to both $q$ and $t$. To avoid a copy of the house graph with vertex set $\{q,r,s,t,u\}$, we must have $t$ adjacent to $u$, but then the subgraph of $G$ induced on $\{p,q,r,s,t,u\}$ is isomorphic to $(K_2+K_1)\vee(K_2+K_1)$, a contradiction.

Thus $G$ cannot contain the configuration in Figure~\ref{fig: configs}(c). Since $\F$ is a self-complementary class of graphs and we observe that any graph is a unigraph if and only if its complement is, $\overline{G}$ cannot contain this configuration, either. Ignoring the vertex labels in Figure~\ref{fig: configs}, we see that if $G$  were to contain the configuration in (d), then $\overline{G}$ would contain the configuration in (c). Thus $G$ contains none of the configurations in Figure~\ref{fig: configs}.
\end{proof}

\begin{cor}
Let $\mathcal{G}$ be a set of graphs. The $\mathcal{G}$-free graphs are all unigraphs if and only if every element of $\F$ induces an element of $\mathcal{G}$.
\end{cor}
\begin{proof}
Since none of the elements of $\F$ is a unigraph, if some element $F$ of $\F$ induces no element of $\mathcal{G}$, then $F$ is $\mathcal{G}$-free and hence the $\mathcal{G}$-free graphs are not all unigraphs. If every element of $\F$ induces an element of $\mathcal{G}$, then the $\mathcal{G}$-free graphs are all $\F$-free and by Theorem~\ref{thm: forb subgr for hered unigraphs} are unigraphs.
\end{proof}

We have characterized the hereditary unigraphs and the DSF sets forbidden for subclasses of these graphs. In conclusion we mention that Tyshkevich gave a structural characterization of general unigraphs in~\cite{Tyshkevich00}. Though beyond the aims of this paper, using either Tyskevich's results or the forbidden subgraph characterization in Theorem~\ref{thm: forb subgr for hered unigraphs}, it is possible to provide a structural characterization (and from thence a degree sequence characterization) for the hereditary unigraphs.

\end{document}